\documentclass[11pt,letterpaper,reqno]{amsart}

\usepackage{amsmath}
\usepackage{amsfonts}
\usepackage{amssymb}
\usepackage{amsthm}
\usepackage{amscd}
\usepackage[hidelinks]{hyperref}
\usepackage{color,soul}
\usepackage{centernot}
\usepackage{enumitem}
\usepackage{todonotes}

\usepackage{bbm}
\usepackage{latexsym}
\usepackage{mathrsfs}
\usepackage{psfrag}
\usepackage[dvips]{epsfig}
\usepackage{epsfig}
\usepackage[all]{xy}         

\theoremstyle{plain}                              

\newtheorem{theorem}{Theorem}
\newtheorem{cor}[theorem]{Corollary}
\newtheorem{proposition}[theorem]{Proposition}
\newtheorem{lemma}[theorem]{Lemma}

\theoremstyle{definition}                         
\newtheorem{definition}{Definition}

\newtheorem{remark}{Remark}

\theoremstyle{remark}                             

\numberwithin{equation}{section}

\newcommand{\R}{\mathbb{R}}                     
\newcommand{\C}{\mathbb{C}}                     
\newcommand{\N}{\mathbb{N}}                     
\newcommand{\prob}{\mathbb{P}}                    
 
\newcommand{\indicator}{\mathbbm{1}} 
\newcommand{\dist}{\mathop{\mathrm{dist}}} 
\newcommand{\Ca}{\mathop{\mathrm{Cap}}} 
\newcommand{\cl}{\mathop{\mathrm{Cl}}}

\newcommand{\eps}{\varepsilon}
\newcommand{\Proba}{\mathcal{P}}

\DeclareMathOperator{\E}{\mathbb{E}}

\newcommand{\emr}[1]{\textcolor{black}{#1}}

\usepackage{mathtools}

\providecommand{\norm}[1]{\left\lVert#1\right\rVert}       

\DeclareMathOperator{\supp}{supp}

\begin{document}

\title{Logarithmic capacity of random $G_\delta$ sets}


\author[F. Quintino]{Fernando Quintino}
\thanks{The author was supported in part by NSF grants DMS-1700143 (PI - M. Foreman) and DMS-1855541 (PI - A. Gorodetski).}

\address{Department of Mathematics, University of California, Irvine}
\email{fquintin@uci.edu}

\subjclass[2010]{Primary: 31A15, 31C15.
  Secondary: 28A12.}
\date{\today}
\dedicatory{}
\keywords{Logarithmic capacity, phase transition, parametric Furstenberg theorem.}
\begin{abstract}
We study the logarithmic capacity of $G_\delta$ subsets of the interval $[0,1].$ Let $S$ be of the form 
\begin{align*}
S=\bigcap_m \bigcup_{k\ge m} I_k,
\end{align*}
where each $I_k$ is an interval in $[0,1]$ with length $l_k$ that decrease to $0$. We provide sufficient conditions for $S$ to have full capacity, i.e. $\Ca(S)=\Ca([0,1])$. We consider the case when the intervals decay exponentially and are placed in $[0,1]$ randomly with respect to some given distribution. The random $G_\delta$ sets generated by such distribution satisfy our sufficient conditions almost surely and hence, have full capacity almost surely. This study is motivated by the $G_\delta$ set of exceptional energies in the parametric version of the Furstenberg theorem on random matrix products. We also study the family of $G_\delta$ sets $\{S(\alpha)\}_{\alpha>0}$ that are generated by setting the decreasing speed of the intervals to $l_k=e^{-k^\alpha}.$ We observe a sharp transition from full capacity to zero capacity by varying $\alpha>0$. 

\end{abstract}

\maketitle


\section{Introduction}
\subsection{The setting}Finite signed Borel measures on $\C$ form a vector space over $\R$. Given two finite singed measures Borel measures $\mu$ and $\nu$ on $\C$ we define their \textit{interaction} by
\begin{align}\label{d.interaction}
I(\nu,\mu):=\iint (-\log|z-w|)\, d\nu(z)\,d\mu(w). 
\end{align}
The interaction is a bilinear form on the vector space of finite singed Borel measures on $\C$ with the following properties: 
\begin{enumerate}
\item $I(\nu,\mu)=I(\mu,\nu),$
\item $I(\nu,\mu)>0$, if $\nu$ and $\mu$ are probability measures and the union of their support has diameter of at most 1,
\item $I(\nu,\mu+\mu')=I(\nu,\mu)+I(\nu,\mu')$ and $I(\nu, c\mu)=c I(\nu,\mu)$.
\end{enumerate}
This bilinear form is a generalization of the \textit{energy} of measure $\mu$, which is defined as $I(\mu):=I(\mu,\mu).$ We can think of the energy of a measure as \textit{self-interacting}. In physics $\mu$ is considered as a charge distribution on $\C$ and $I(\mu)$ as the total energy of $\mu$ on $\C$ (see \cite[pg. 56]{Ra}).

The \textit{logarithmic capacity} of a subset $E\subset \C$ is then defined by minimizing the energy:
\begin{align*}
\Ca(E) = \exp(-\inf\{ I(\mu)\}),
\end{align*}
where the infimum is taken over the set of Borel probability measures whose support is a compact subset of~$E$ (we interpret $e^{-\infty}$ as $0$). 
Capacity gauges how far away a set is from being \textit{polar}. A set $E$ is said to be \textit{polar} if $I(\mu)=\infty$ for every non-trivial measure $\mu$ with compact support in $E$.

Most of the literature on capacity has been devoted to the study of compact sets in $\C$, but non-compact sets has also been studied (see \cite{Ra}, \cite[Appendix A]{Si}). Moreover, $G_\delta$ sets have also been of interest (see \cite{De}, \cite{KQ}, \cite{Ra}). \emr{In \cite[Section 9]{Da}, the authors discuss the applications of potential theory to spectral theory.}

Our focus will be on the capacity of $G_\delta$'s of the form:
\begin{align}\label{d.S}
S=\bigcap_m \bigcup_{k\ge m} I_k,
\end{align}
where each $I_k$ is an open interval of length $l_k$ with center at $c_k\in (0,1)$. The sequence $\{l_k\}$ is taken to approach 0 as $k\rightarrow \infty$. It is immediate that $S$ is a $G_\delta$ subset of $[0,1]$. Under certain assumptions, we will show that the set $S$ has {full capacity} on the unit interval:
\begin{definition}
Let $J\subset\C$. A set $E\subset J$ is said to have \textit{full~capacity} on $J$ if
$$\Ca(E)=\Ca(J).$$
\end{definition}
The capacity of an interval $J$ is $\Ca(J)=\frac{|J|}{4}$ (see, e.g.~\cite[pg.~135]{Ra},~\cite[Example~A.17]{Si}).

Our results and methods can be extended to higher dimensions, but we do not elaborate on that here. In this paper, we will be focused on $G_\delta$ subsets of an interval of the real line. We are mostly interested in one-dimension because our motivation came from the one-dimensional random $G_\delta$ sets of exceptional energies in the parametric version of the Furstenberg theorem (see Section \ref{s.motivation}).

\subsection{Main results}\label{s.main.results}
Our first main result is devoted to the random setting, being a ``toy model" for the exceptional energies in the parametric version of the Furstenberg theorem (see Section \ref{s.motivation}). Namely, the set of exceptional parameters in \cite{GK} is generated by exponentially small intervals, that are asymptotically distributed with respect to some (dynamically defined) measure. However, their positions are random.

A \textit{random} $G_\delta$ set is obtained by viewing the centers $\{c_k\}$ as random variables. As a toy model, it is reasonable to consider first the set generated by random intervals, that are placed independently (with the same ``reasonable" distribution of their centers) - instead of some complicated definition coming from the random dynamical systems.

\begin{theorem}\label{t.random.full}
Let $\{c_k\}$ be i.i.d. with an absolutely continuous distribution on any interval $J$ with almost everywhere positive and uniformly bounded density function. Take $l_k=e^{-\lambda k}$ for some fixed $\lambda>0$ and let $S$ be the corresponding $G_\delta$-set \eqref{d.S}. Then, almost surely $S$ has full capacity on the unit interval:
\begin{align*}
 \Ca(S)=\Ca([0,1])>0.
\end{align*}

\end{theorem}

\begin{remark}
Full capacity is a property that is inherited when restricted to subintervals (see {\cite[Proposition 1.6]{KQ}}): If $E$ is a subset of interval $J$ such that $\Ca(E)=\Ca(J)$, then given any subinterval $J'\subset J$, one has $\Ca(E\cap J')=\Ca(J')$.
\end{remark}
\begin{remark}\label{r.interval.WLOG}
In Theorem \ref{t.random.full} (and Theorem \ref{t.e^-k.fullcap} below), the interval $[0,1]$ may be replaced with any bounded interval $J$ due to the fact that $\Ca( \beta\cdot J)= \beta\cdot\Ca([0,1])$ for some $\beta>0$. Without loss of generality, we will only be working on the interval $[0,1]$. So, we will take $dx$ to be the restriction to the unit interval: $\left. dx\right|_{[0,1]}$.
\end{remark}

Now, take the centers $\{c_k\}$ from Theorem \ref{t.random.full} and let us vary the lengths of the intervals as a function of the parameter $\alpha\in(0,1]$: $l_k=e^{-\lambda k^\alpha}$. It turns out that at $\alpha=1$ the capacity undergoes a (sharp) phase transition:

\begin{theorem}[Random phase transition]\label{t.random.phase.transition}
Let $\{c_k\}$ be i.i.d. with an absolutely continuous distribution with density function that is bounded and positive almost everywhere with respect to the Lebesgue measure. Let $S$ be generated by $l_k:= e^{-\lambda k^\alpha}$ for $\lambda>0$ and $\alpha>0$. Then
\begin{enumerate}
\item $\Ca(S)=\Ca([0,1])>0$ for $0<\alpha\leq 1$ 
almost surely,
\item $\Ca(S)=0$ for $1<\alpha$.
\end{enumerate}
\end{theorem}

\begin{remark}\label{r.conjecture}
The phase transition in Theorem \ref{t.random.phase.transition} (and in Theorem  \ref{t.second.phase.transition} below) is analogous to the one observed in \cite{KQ} (see Equation \eqref{d.uniform.S} and Theorem \ref{t.Phase.transition} in Section \ref{s.motivation}). It is interesting to note that in the present paper we actually establish the full capacity at the critical point $\alpha=1,$ while for the setting in \cite{KQ} full capacity at the critical value $\alpha=2$ was only a conjecture.
\end{remark}

Theorem \ref{t.random.full} and Theorem \ref{t.random.phase.transition} follow from our deterministic results. Our first main deterministic result provides sufficient conditions for a $G_\delta$ set $S$, defined by \eqref{d.S}, to have full capacity on the unit interval:

\begin{theorem}[Sufficient conditions for full capacity]\label{t.e^-k.fullcap}
Assume that the intervals $\{I_k\}$ from \eqref{d.S} have exponentially decreasing lengths $l_k=e^{-\lambda k}$ for some fixed $\lambda>0$ and satisfy the assumptions \ref{a.distributed} - \ref{a.uniform} below. Then the $G_\delta$ set $S$ has full capacity on the unit interval:
\begin{align*}
 \Ca(S)=\Ca([0,1])>0.
\end{align*} 

\end{theorem}

The assumptions that are imposed in this theorem, roughly speaking, state that these intervals are sufficiently  uniformly placed and sufficiently well-spaced (both in terms of ``average" and minimal distances between their centers).

First, we will pack intervals $\{I_k\}$ into groups with the indices from 
\begin{align}\label{d.A_n}
\mathcal{A}_n:=\{n,\ldots,2n-1\},
\end{align}
 and then pack these into larger groups:
 \begin{align}\label{d.A_n,q}
\mathcal{A}_{n,q(n)}:=\mathcal{A}_n\cup \mathcal{A}_{2n}\cup \cdots\cup \mathcal{A}_{2^{q}n},
\end{align}
where $q\in\N$. We will assume the following:

\begin{enumerate}[label=\textbf{A.\arabic*}]
\item \label{a.distributed} \textbf{(distribution)} The centers are \textit{distributed} with respect to some density function $\phi(x)\in L^1([0,1],\,dx),$ where $\phi(x)>0$ a.e.. Namely, for every $f\in C([0,1])$, we have  
\begin{align*}
\frac{1}{\#(\mathcal{A}_{n})}\sum_{k\in\mathcal{A}_n}f(c_k)\to \int_0^1 f(x)\phi(x)\,dx \quad \text{as}\quad n\rightarrow \infty.
\end{align*}
\end{enumerate}
Also, there exists a sequence $q(n)$ of integer numbers, such that $q(n)\to\infty$ as $n\to\infty$ and that the following two conditions hold:
\begin{enumerate}[label=\textbf{A.\arabic*}]\addtocounter{enumi}{1}
\item \label{a.summable} \textbf{($\log$-average spacing)} For every $\eps>0$, there exists $\delta>0$ such that for all $n$ large enough for any $n',n''\in \{n,2n,\dots, 2^{q(n)}n \}$ we have
\begin{align*}
\frac{1}{\#(\mathcal{A}_{n'}\times \mathcal{A}_{n''})}\sum(-\log|c_i-c_j|)<\eps,
\end{align*} 
where the sum is over $i\in\mathcal{A}_{n'}, j\in\mathcal{A}_{n''}$ such that $i\neq j$ and $|c_i-c_j|<\delta$.
\item \label{a.uniform}\textbf{(gap control)} For every $\eps>0$, for all large enough $n$, we have that for every $i,j\in \mathcal{A}_{n,q(n)} $ with $i\neq j$ the following holds:
\begin{align*}
\frac{l_{i}+l_j}{2|c_i-c_j|}<\eps.
\end{align*}

\end{enumerate}


Applying Theorem \ref{t.e^-k.fullcap}, we obtain our next result, a deterministic phase transition for the capacity. Again, take the centers $\{c_k\}$ that assumptions \ref{a.distributed}-\ref{a.uniform} are satisfied for the choice of lengths $l_k=e^{-\lambda k}$ for some fixed $\lambda>0$ and varying the speed at which the lengths of intervals decrease, we observe a sharp phase transition in the deterministic setting:

\begin{theorem}[Deterministic phase transition]\label{t.second.phase.transition}
Let the centers $\{c_k\}$ be the same centers from Theorem \ref{t.e^-k.fullcap}. Let $S$ be generated by $l_k:= e^{-\lambda k^\alpha}$ for $\alpha>0$. Then
\begin{enumerate}

\item $\Ca(S)=\Ca([0,1])>0$ for $0<\alpha\leq 1$,
\item $\Ca(S)=0$ for $1<\alpha$.

\end{enumerate}
\end{theorem}

The next theorem states that both the random theorems (Theorem \ref{t.random.full} and Theorem \ref{t.random.phase.transition}) follow from the deterministic theorems (Theorem \ref{t.fullcapacity} and Theorem \ref{t.second.phase.transition}):

\begin{theorem}\label{t.random.assumptions}
Let $\{c_k\}$ be i.i.d. with an absolutely continuous distribution on any interval $J$ with almost everywhere positive and uniformly bounded density function. Take $l_k=e^{-\lambda k}$ for some fixed $\lambda>0$. Then, almost surely assumptions \ref{a.distributed}-\ref{a.uniform} are satisfied.
\end{theorem}

\subsection{Motivation and historical background}\label{s.motivation}


In this section, we will discuss the motivation behind our project and the historical background. 

Gorodetski and Kleptsyn in \cite[Section 1.2]{GK} studied the set of exceptional energies in the parametric version of the Furstenberg theorem. Consider
$$
T_{n,\omega,a}:=A_{\omega_n}(a)\dots A_{\omega_1}(a)
$$
where matrices~$A_{\omega_k}(a)\in SL(2,\R)$ are i.i.d., depending on a parameter $a$, taking values in some interval $J\subset \R$. Furstenberg's theorem implies that for every $a\in J$, for almost every $\omega$, we get
\begin{align}\label{e.furstenberg.limit}
\lim_{n\to\infty} \frac{1}{n} \log \|T_{n,\omega,a}\| = \lambda_F(a)>0.
\end{align}

Questions on switching the quantifiers in the limit appear naturally \emr{in spectral theory, specifically, in Anderson localization proofs.}

In~\cite[Theorem~1.5]{GK}, the authors proved that almost surely switching the quantifiers leads to the occurrence of a different kind of behavior. Namely, under some technical assumptions, it was shown that for almost every $\omega$, there exists some random \textit{exceptional energies} subset of parameters $S_e(\omega)\subset J$ such that \eqref{e.furstenberg.limit} does not hold. Additionally, there also exists a smaller set of parameters $G_{\delta}$-set $S_0(\omega)$ such that for all $a\in S_0(\omega)$, we have
$$
\liminf_{n\to\infty} \frac{1}{n} \log \|T_{n,\omega,a}\| =0.
$$
Both these sets are random $G_\delta$'s of the form \eqref{d.S}. 

Additionally, in \cite{GK} it was shown that the set $S_e(\omega)$ (and thus $S_0(\omega)$) have zero Hausdorff dimension. Capacity is a finer measurement than the Hausdorff dimension in the sense that any set $E\subset \C$ that has zero capacity must have zero Hausdorff dimension. The question as to what is the capacity of both $S_e(\omega)$ and $S_0(\omega)$ is still open. If one can show that those sets satisfy assumptions \ref{a.distributed}-\ref{a.uniform} (and this is what we conjecture), our Theorem \ref{t.e^-k.fullcap} will imply that these sets have full capacity, that is $\Ca(S_e(\omega))=\Ca(S_0(\omega))=\Ca(J)$,
in the same way as we get full capacity in the ``toy model" Theorem~\ref{t.random.phase.transition}.

The capacity of such $G_\delta$'s is also interesting because it showcases a phase transition. That is, a drastic transition from zero capacity to full capacity precisely when the series 
\begin{align}\label{e.log.sum}
\sum_k \frac{1}{|\log l_k|}
\end{align}
transitions from convergent to divergent. As we mentioned above, capacity gauges how far away a set is from being polar. Hence, as we change the speed of intervals so that the series \eqref{e.log.sum} transitions from convergent to divergent, $S$ goes from being polar to being as far away as possible from polar, there is no middle ground. This transition was first noticed by Kleptsyn and Quintino (see \cite{KQ}) in the case when the centers $\{c_k\}$ are equidistributed in the following way: for every $n$ we consider $n$ equally spaced centers:
$$c_{j,n} =\frac{2j+1}{2n} \quad \text{for every} \quad j=0,\ldots,n-1,$$ 
and with the restriction that the corresponding interval $J_{j,n}$ have the same length $r_n$ for $j=0,\ldots,n-1$. The \emph{uniform $G_\delta$-set}~$\widetilde{S}$, corresponding to the sequence~$r_n$, is given by
\begin{equation}\label{d.uniform.S}
\widetilde{S}:=\bigcap_{m=1}^\infty\bigcup_{n=m}^\infty  \bigcup_{j=0}^{n-1} J_{j,n}.
\end{equation}
Any uniform $G_\delta$ set $\widetilde{S}$ may be written in the generic setting \eqref{d.S} by ordering $J_{j,n}$ and re-labeling. They noticed that there is a ``phase transition" in which $\widetilde{S}$ goes from having zero capacity to {full capacity}:
\begin{theorem}[Phase transition {\cite[Theorem 1.2]{KQ}}]\label{t.Phase.transition}
For $r_n=e^{-n^\alpha}$,
\begin{enumerate}
\item if $\alpha >2$, then $\Ca(\widetilde{S})=0$,
\item if $\alpha <2$, then $\Ca(\widetilde{S})=\Ca([0,1])$.
\end{enumerate}
\end{theorem}

We refer to $\alpha=2$ as the critical case because it is precisely when the sum~\eqref{e.log.sum} transitions from convergent to divergent. Note that there are $n$ intervals of length $e^{-n^{\alpha}}$ in~\eqref{d.uniform.S}, and that is why the critical case is $\alpha=2$ in Theorem \ref{t.Phase.transition} and not $\alpha=1$. Also, note that full capacity of $\widetilde{S}$ in the critical case in \cite{KQ} was \emph{conjectured}, but not proved; contrary to this, in the setting of the present paper the analogous statement for the critical $\alpha=1$ is \emph{established} (see Remark~\ref{r.conjecture}). 

The zero capacity part in all the theorems above goes back to the works in the first half of twentieth century: a 1918 paper by Lindeberg~\cite{Lin} and 1937 by Erd\"os and Gillis~\cite{E-G}. They were working on connecting the notion of the $h$-\emph{volume} of a set with the \textit{logarithmic capacity} of a set. A function $h$ that is defined in some right neighborhood of $0$ is called a \emph{measuring function} provided that $h$ is continuous, positive, increasing, concave, and $h(0)=0$. The $h$-\emph{volume} of a set $E\subset \R$ is defined as
$$
m_h(E):=\lim_{\eps \to 0+} \, \inf_{\{(x_j,r_j)_{j\in\N}\} \in \mathcal{I}(E,\eps)} \sum_j h(r_j),
$$
where the infimum is taken over the set $\mathcal{I}(E,\eps)$ of covers of $E$ by balls of diameter less than $\eps$:
$$
\mathcal{I}(E,\eps) = \left\{(x_j,r_j)_{j\in \N} \mid \bigcup_j U_{r_j}(x_j)\supset E, \quad \forall j \quad r_j<\eps \right\}.
$$
In particular, the function $h_0(x)=\frac{1}{|\log x|}$ provided the following link: 
\begin{theorem}[\mbox{Erd\"os and Gillis~\cite[p.~187]{E-G}, generalizing Lindeberg~\cite[p.~27]{Lin}}]\label{c.zero.cap}

If for a set E one has $m_{h_0}(E) < +\infty$, then $\Ca(E) = 0$.
\end{theorem}
This was later re-proved by Carleson~\cite{Car}, and noticed in~\cite[Thm.~1.3]{KQ} to be a corollary of the Cauchy-Schwarz inequality. Theorem~\ref{c.zero.cap} immediately implies (see~\cite[Corollary~1.4]{KQ}):
\begin{cor}
Let $S$ be defined by \eqref{d.S}. 
If the series $\sum_k \frac{1}{|\log l_k|}$ converges, then the set ${S}$ is of zero capacity.
\end{cor}

In the same 1937 paper, Erd\"os and Gillis~\cite[(C), p. 186]{E-G} have mentioned a conjecture, going back to Nevanlinna's paper~\cite{Nev}, that aimed at generalizing Theorem~\ref{c.zero.cap} to other $h$-volume settings. This conjecture was disproved by Ursell \cite{Urs}; the re-distribution construction that was used in~\cite{KQ} and that we are using in the present paper can be seen as an extension of his technique.

%


\subsection{Sketch of the proof and plan of the paper}\label{s.plan} In this section, we will give a sketch of the proofs and end with the plan of the paper.

The statement in the phase transition theorems \ref{t.random.phase.transition} and \ref{t.second.phase.transition} for  $\alpha>1$ is a result from {\cite{KQ}} and does not require assumptions \ref{a.distributed} - \ref{a.uniform} (see Theorem \ref{c.zero.cap} above).

Due to monotonocity of capacity, the statement in the phase transition theorems \ref{t.random.phase.transition} and \ref{t.second.phase.transition} for $0<\alpha<1$ follows by establishing full capacity for $\alpha=1$. For the deterministic phase transition this is Theorem \ref{t.e^-k.fullcap}. For the random phase transition the result follows from Theorem \ref{t.e^-k.fullcap} by showing that assumptions \ref{a.distributed}-\ref{a.uniform} hold almost surely, that is Theorem \ref{t.random.assumptions}.

In Section \ref{s.random.phase.transition}, we will
show that the centers from the random phase transition satisfy assumptions \ref{a.distributed}~-~\ref{a.uniform} for $\alpha=1$ (Theorem \ref{t.random.assumptions}). Hence, the random phase transition holds. 

Thus, the main task is to show that $S$ has full capacity for $\alpha=1$ (Theorem \ref{t.e^-k.fullcap}). The method that we will employ to show full capacity is the \emph{re-distribution technique} under assumptions \ref{a.distributed}-\ref{a.uniform} for $\alpha=1$. 

We introduce this technique in Section~\ref{s.fullcapacity}. Namely, we will begin with the
equilibrium measure $\nu_J$ on the interval $J$, then we will construct a probability measure $\nu_1$  such that the energy $I(\nu_1)$ approximates the energy $I(\nu_J)$ and whose support is a subset of $\supp\nu_J$ and is a finite union of intervals $\{I_k\}$. Then we will construct another probability measure $\nu_2$  such that the energy $I(\nu_2)$ approximates the energy $I(\nu_1)$ and whose support is a subset of $\supp\nu_1$ and is a finite union of intervals $\{I_k\}$. Inductively, repeating this procedure, we get a sequence of probability measures that have their energies that are arbitrarily close to $I(\nu_J)$ and such that their supports create a decreasing sequence of compact subsets. After passing to the weak-limit we obtain a measure supported on $S$ (Proposition \ref{t.fullcapacity}), thus proving the desired full capacity for the set $S$ (see Section \ref{s.re-dist.proof} for the proof). Proposition \ref{p.S.alpha=1.total.energy} states that the above technique is applicable when assumptions \ref{a.distributed}-\ref{a.uniform} are satisfied. Hence, Theorem~\ref{t.e^-k.fullcap} follows from Proposition~\ref{t.fullcapacity} and Proposition~\ref{p.S.alpha=1.total.energy}.

Finally, in Section \ref{s.main.proof} we develop the tools to prove Proposition~\ref{p.S.alpha=1.total.energy}. In Section~\ref{s.critical} we conclude with the proof of Proposition~\ref{p.S.alpha=1.total.energy}.


\section{In the random setting \ref{a.distributed}-\ref{a.uniform} are a.s. satisfied}
\label{s.random.phase.transition}

This section is devoted to the proof of Theorem~\ref{t.random.assumptions}: we assume that 
the centers $\{c_k\}$ are i.i.d. random variables and show that if their distributions are nice, then assumptions~\ref{a.distributed}-\ref{a.uniform} are satisfied.

The distribution immediately follows from the law of large numbers:
\begin{lemma}\label{l.A1sat}
Under the assumptions of Theorem~\ref{t.random.assumptions}, assumption~\ref{a.distributed} is almost surely satisfied.
\end{lemma}

Now, take $q(n)=[\log_2(\log n)]$. The uniform gap control can be obtained by a straightforward estimate of the probability of two random centers being close to each other:
\begin{lemma}\label{l.A3sat}
Under the assumptions of Theorem~\ref{t.random.assumptions}, for $q(n)=[\log_2(\log n)]$, assumption~\ref{a.uniform} is almost surely satisfied.
\end{lemma}
\begin{proof}
Let $K$ be the upper bound for the density of the distribution, and let $\eps>0$ be fixed. For any $i\neq j$, $i,j\in \mathcal{A}_{n,q(n)} $, if
\[
\frac{l_i + l_j}{2 | c_i-c_j|}< \eps
\]
does not hold, it implies that 
\[
|c_i-c_j| \le  \frac{l_i + l_j}{2\eps} < \frac{1}{\eps} e^{-\lambda n},
\]
and the probability of such an event (for any given $i$ and $j$) does not exceed~$\frac{2K}{\eps} e^{-\lambda n}$. As there are less than $2^{q(n)+1}n<2n^2$ possible indices $i$ and $j$, the total probability that the condition is violated for a given $n$ does not exceed $4n^4\cdot \frac{2K}{\eps} e^{-\lambda n}$. The series 
\[
\sum_n 4n^4\cdot \frac{2K}{\eps} e^{-\lambda n}
\]
converges, and the application of the Borel-Cantelli Lemma concludes the proof.
\end{proof}

Finally, the log-averages of spaces also can be controlled quite directly: 
\begin{lemma}\label{l.A2sat}
Under the assumptions of Theorem~\ref{t.random.assumptions}, for $q(n)=[\log_2(\log n)]$, assumption~\ref{a.summable} is almost surely satisfied.
\end{lemma}
\begin{proof}

Given $\eps>0$ be given and set 
$$G(X,Y)=(-\log|X-Y|)\indicator_{(0,\delta)}(|X-Y|),$$
where $\delta>0$ and $G(X,X)=0$. We have that 
\begin{align*}
\frac{1}{\#(\mathcal{A}_{n'}\times \mathcal{A}_{n''})}\sum_{0<|c_i-c_j|<\delta}(-\log|C_i-C_j|)=\frac{1}{\#(\mathcal{A}_{n'}\times \mathcal{A}_{n''})}\sum_{i\in\mathcal{A}_{n'}, j\in\mathcal{A}_{n''} }G(C_i,C_j).
\end{align*} 

Suppose the law of large numbers holds for $G(X,Y)$: as $n\rightarrow\infty$ we have
\begin{align*}
\frac{1}{\#(\mathcal{A}_{n'}\times \mathcal{A}_{n''})}\sum_{i\in\mathcal{A}_{n'}, j\in\mathcal{A}_{n''} }G(C_i,C_j)\rightarrow \E G(C_1,C_2),
\end{align*}
where  for $n',n''\in \{n,2n,\dots, 2^{q(n)}n \}$. 
Then we may find a $\delta$ such that \ref{a.summable} holds.

The law of large numbers holds by considering the difference:
\[
H(x,y):=G(x,y)-c-\E[ G(x,y)-c|y]-\E[ G(x,y)-c|x],
\] 
where $c=\E G(x,y)$
and their average
\begin{align}\label{e.average}
\frac{1}{n'n''} S_{n',n''}=\frac{1}{n'n''}\sum_{i\neq j}H(x,y).
\end{align}
Now, consider the fourth power of $S_{n',n''}$ and take its expectation:
\begin{align}\label{e.sum.expectation}
\E (S_{n',n''})^4=\sum \E [H(c_{i_1},c_{i_2})H(c_{j_1},c_{j_2})H(c_{k_1},c_{k_2})H(c_{l_1},c_{l_2})].
\end{align}
The function $H(x,y)$ has the property that $\E[ H(x,y)|y]=\E[ H(x,y)|x]=0$.
Hence, if a term has an independent random variable, say $c_{i_1}$, then  
\begin{align*}\label{e.sum.expectation}
\E [H(c_{i_1},c_{i_2})H(c_{j_1},c_{j_2})H(c_{k_1},c_{k_2})H(c_{l_1},c_{l_2})]=0.
\end{align*}
On the other hand, when every random variable is depended on another random variable, we can count the non-vanishing terms. There are $(n'n'')$ terms of the form $\E [H(c_{i_1},c_{i_2})^4].$ There are $(n'n'')^2$ terms of the form
$$\E [H(c_{i_1},c_{i_2})^2H(c_{j_1},c_{j_2})^2].$$
There are at most $(n'n'')(n'+n'')4$ terms of the form
$$\E [H(c_{i_1},c_{i_2})^2H(c_{j_1},x)H(c_{k_1},y)],$$
where $x,y\in\{c_{i_1}\,c_{i_2}\}.$
 Lastly, there are at most $n'n''(n'+n'')^2$ terms of the form
$$\E[H(c_{i_1},c_{i_2})H(c_{i_1},c_{j_2})H(c_{k_1},c_{i_2})H(c_{k_1},c_{j_2})].$$
Since $\E[H(x,y)^4]<\infty$, then
\[
\E S_{n',n''}^4 \le C'\max\{(n'n'')^2, n'^3n'',n'n''^3\},
\]
where $C'>0$ is some constant. An application of the Chebyshev inequality implies that
\begin{align*}
\prob\left( |S_{n',n''}|>\eps (n'n'') \right)&\leq \E(S_n)^4/(\eps (n'n''))^4\\
&\leq \frac{C'}{\eps^4}\max\left\lbrace\frac{1}{(n'n'')^2},\frac{1}{n'n''^3},\frac{1}{n'^3n''}\right\rbrace.
\end{align*}
Since $n\leq n',n''$, then 
\begin{align*}
\prob\left( |S_{n',n''}|>\eps (n'n'') \right)&\leq \frac{C'}{\eps^4}\frac{1}{n^4}.
\end{align*}
We have that 
\begin{align*}
\sum_{n=1}^\infty\sum_{n'=n}^{2^{q(n)}n}\sum_{n''=n}^{2^{q(n)}n}\prob\left( |S_{n',n''}|>\eps (n'n'') \right)&\leq \frac{C'}{\eps^4}\sum_{n=1}^\infty \frac{2^{2q(n)}}{n^2}\\
&\leq \frac{C'}{\eps^4}\sum_{n=1}^\infty \frac{(\log n)^2}{n^2},
\end{align*}
which is finite. By Borel-Cantelli lemma, $ |S_{n',n''}|>\eps (n'n'')$ does not occur infinitely often with probability 1. Let $\eps_k$ be a sequence of positive numbers that decreases to $0$ as $k\rightarrow\infty$. For each $\eps_k$, $ |S_{n',n''}|>\eps_k (n'n'')$ does not occur infinitely often with probability 1. Since the countable intersection of sets of full measure has full measure, then for all $\eps>0$, there exists $n_0\in\N$ such that for any $n\geq n_0$, for every $n',n''\in \{n,2n,\dots, 2^{q(n)}n \}$, we have $ |S_{n',n''}|<\eps (n'n'')$ with probability 1. That is, the average \eqref{e.average} goes to $0$ as $n\rightarrow\infty$.

\end{proof}

Together, lemmas \ref{l.A1sat}, \ref{l.A3sat}, \ref{l.A2sat}  imply Theorem~\ref{t.random.assumptions}.




\section{The re-distribution technique}\label{s.re-distribution-technique}

\subsection{Introducing the technique}\label{s.fullcapacity}
Our main tool for establishing full capacity for a set $S$ (what is needed for the proof of Theorem \ref{t.e^-k.fullcap}) will be the \emph{re-distrubtion technique} that was introduced in \cite{KQ}. We will recall the method in this section. The main property that allows it to work will be the following one:

\begin{definition}\label{d.TEA}
We say that $S$ (in the generic setting \eqref{d.S}) is \textit{re-distributable} if the following holds: for every probability measure $\nu$ with piecewise continuous density that is supported on a finite collection of intervals in $[0,1]$ and for every $\eps>0$ and every $m\in\N$, there exists another probability measure $\nu'$ with piecewise continuous density such that 
\begin{enumerate}
\item $I(\nu')<I(\nu)+\eps$,
\item $\nu'$ is supported on $\supp \nu \cap V_n$ for some $n\geq m$,
\end{enumerate}
where $V_n$ is a finite union of $I_k$'s with $k\geq n$.
\end{definition}

The following proposition then allows us to establish full capacity:

\begin{proposition}\label{t.fullcapacity} If  $S$ is re-distributable, then $S$ has full capacity on the unit interval:
$$\Ca (S)=\Ca ([0,1]).$$
\end{proposition}

\begin{proposition}\label{p.S.alpha=1.total.energy} Assume \ref{a.distributed} - \ref{a.uniform} for interval lengths $l_k=e^{-\lambda k}$ for some $\lambda>0$. Then the set $S$ is re-distributable.
\end{proposition}

Section \ref{s.main.proof} is devoted to the proof of Proposition \ref{p.S.alpha=1.total.energy}. 


\subsection{Proof of Proposition \ref{t.fullcapacity}}\label{s.re-dist.proof}
In this section, we will prove that $S$ has full capacity when $S$ is re-distributable.

As we have mentioned in Section \ref{s.plan}, the proof of Proposition \ref{t.e^-k.fullcap} is obtained by inductively constructing a sequence of measures with smaller and smaller support. Let us make these arguments formal:

\begin{proof}[Proof of Proposition \ref{t.fullcapacity}]

The density function for the equilibrium measure for the unit interval is
\begin{align*}
f_{[0,1]}(x)=\frac{1}{\pi \sqrt{x(1-x)}},
\end{align*}
for $x\in(0,1)$ and $0$ otherwise (see e.g.~\cite[Eq.~(A.53)]{Si}). 
Given $\eps>0$, there exists a continuous density function $f$ such that
$$I(f(x)dx)<I(f_{[0,1]}(x)dx)+\eps.$$
Let $d\nu_0(x):=f(x)\,dx$ with support $[0,1]$. Applying Definition \ref{d.TEA} to $\nu_0$, there exists $\nu_1$ with support $V_{n_1}$ and 
$$I(\nu_1)<I(\nu_0)+\eps/2^2.$$
Apply Definition \ref{d.TEA} to $\nu_1$, there exists $\nu_2$ with support $V_{n_1}\cap V_{n_2}$ and 
$$I(\nu_2)<I(\nu_1)+\eps/2^3.$$ and $n_1<n_2$.
By induction and applying Definition \ref{d.TEA}, for each $m\in\N$ there exists a Borel probability measure $\nu_m$ that is supported on $$C_m:=V_{n_1}\cap\cdots\cap V_{n_m}.$$ 
We consider the telescoping sum:
\begin{align*}
I(\nu_m)-I(\nu_0)=\sum_{i=1}^m \big(I(\nu_i)-I(\nu_{i-1})\big)<\sum_{i=1}^m\frac{\eps}{2^{i+1}}.
\end{align*} 
It follows that 
\begin{align*}
I(\nu_m)<I(\nu_0) + \eps.
\end{align*}
As in \cite{KQ}, any weak* limit will work. Assume that $\nu_\infty$ is a weak* limit of $\{\nu_m\}.$ Passing to a weak* limit can only decrease the energy (see \cite[Lemma 3.3.3]{Ra}): 
\begin{align*}
I(\nu_\infty)\leq\liminf_{m\rightarrow\infty} I(\nu_m)<I(\nu_0) + \eps< I(f_{[0,1]}(x)\,dx)+2\eps.
\end{align*}
Since $\eps>0$ is arbitrary, we have that 
\begin{align*}
I(\nu_\infty)\leq I(f_{[0,1]}(x)\,dx).
\end{align*}
If $\nu_\infty$ has compact support contained in $S$, then we are done. The weak* limit only allows us to conclude that $\nu_\infty$ has compact support contained in
\begin{align*}
 \bar{C}_\infty:=\bigcap_m \cl(C_m).
\end{align*}
However, $\bar{C}_\infty$ differs from
\begin{align*}
 C_\infty:=\bigcap_m C_m \subset S,
\end{align*}
by at most a countable set $P$ (the collection of boundary points of each $C_m$). Since $I(\nu_\infty)<\infty$, then $\nu_\infty (P)=0$ (see \cite[Theorem 3.2.3]{Ra}). By regularity of Borel measures, we may find a Borel probability measure with compact support contained in $C_\infty$ that differs from $I(\nu_\infty)$ as small as we want. Hence, 
\begin{align*}
I(f_{[0,1]}(x)\,dx)=\inf\{I(\nu):\nu\in\Proba(C_\infty)\}.
\end{align*}
Since $C_\infty\subset S$, then $S$ has full capacity:
\begin{align*}
\Ca(S)=\Ca([0,1]).
\end{align*}

\end{proof}


\section{Proving Proposition \ref{p.S.alpha=1.total.energy}: \ref{a.distributed}-\ref{a.uniform} imply re-distribution }\label{s.main.proof}



\subsection{Properties of assumptions \ref{a.distributed}-\ref{a.uniform}}\label{s.assumptions}
In this section, we will discuss some of the properties of assumptions \ref{a.distributed}-\ref{a.uniform} that will be needed in the proofs.

The gap control property (assumption \ref{a.uniform}) is aimed at controlling the gaps between two distinct intervals in a collection of intervals in a uniform way. Let $I,I'\subset (0,1)$ be two disjoint intervals with centers $c,c'$, then the gap between $I$ and $I'$ is
\begin{align}\label{e.gap.distance}
\dist(I,I')=|c-c'|-\frac{1}{2}(|I|+|I'|)>0.
\end{align}
We will control the gaps by controlling the ratio of the average of the lengths and the distance between their centers (see \ref{a.uniform}).

\begin{remark}\label{r.I_k.are.disjoint}
Notice that by letting $\eps<1$ in \ref{a.uniform}, we get \eqref{e.gap.distance}. Hence, gap control implies that the intervals in
\begin{align*}
\{I_k: k\in 2n,\ldots, 2^{q(n)}n-1\},
\end{align*}
are pairwise disjoint.
Each measure that we construct in Section \ref{s.construction} will be supported on pair-wise disjoint collection:
\begin{align*}
\{I_k: k\in\mathcal{A}_n\}=\{I_n,I_{n+1},\ldots, I_{2n-1}\}.
\end{align*}
In Section \ref{s.critical}, we will construct a measure that is an average of measures from Section \ref{s.construction}. Hence, the average measure will be supported on: 
\begin{gather*}
\{  I_n,I_{n+1},\ldots, I_{2n-1}\}\\
\{I_{2n},I_{2n+1},\ldots, I_{2^2n-1}\}\\
\vdots\\
\{I_{2^{q(n)-1}n},I_{2^{q(n)-1}n+1},\ldots, I_{2^{q(n)}n-1}\}.
\end{gather*}
Assumption \ref{a.uniform} allows the collection of intervals above to be disjoint.
\end{remark}

The distribution property (assumption \ref{a.distributed}) requires the centers to be \textit{distributed} with respect to some function $\phi(x)\in L^1([0,1],\,dx),$ where $\phi(x)>0$ a.e.:
\begin{align}\label{d.distributed.1d}
\frac{1}{\#(\mathcal{A}_{n})}\sum_{k\in\mathcal{A}_n}f(c_k)\to \int_0^1 f(x)\phi(x)\,dx \quad \text{as}\quad n\rightarrow \infty,
\end{align}
for every continuous function $f\in C([0,1])$. This definition is a generalization of equidistributed sequences.
\begin{remark}
Note that \eqref{d.distributed.1d} will hold for piecewise continuous functions $f$ since we may approximate such functions from above and below by continuous functions in $L_1$. Equation \eqref{d.distributed.1d} extends to 2-dimensions: Let $f$ be any piecewise continuous function. Then
\begin{align}\label{d.distributed.2d}
\frac{1}{\#(\mathcal{A}_{n}\times \mathcal{A}_{n'})}\sum_{i\in\mathcal{A}_{n}, j\in\mathcal{A}_{n'} }f(c_i)f(c_j)\to \int_0^1\int_0^1 f(x)f(y)\phi(x)\phi(y)\,dxdy,
\end{align}
as $m\rightarrow\infty$ and $n,n'\geq m$.

\end{remark}

To show that $S$ is re-distributable (see Definition \ref{d.TEA}), we will show that the following statement holds:
\begin{enumerate}[label=\textbf{P.\arabic*}]
\item \label{property} For each positive continuous function $f$ on the interval $[0,1]$, there exists a sequence of probability measures $\{\mu^n\}$ so that each $\mu^n$ has a piecewise continuous density with
support contained in a finite union of disjoint $I_k$'s with $k\geq n$ and with asymptotic behavior:
\begin{align*}
I(\mu^n)=& \frac{I(f(x)\, dx)}{(\int_0^1 f(x)\,dx)^2} +o(1).
\end{align*}

\end{enumerate}

With the distribution assumption \ref{a.distributed} and $\log$-average spacing assumption \ref{a.summable}, we can see that the centers have the asymptotic behavior that is needed in \ref{property}:

\begin{lemma}\label{l.SLLN.non.random}
Under assumptions \ref{a.distributed} and \ref{a.summable}, for every $f\in C([0,1])$, as $n\rightarrow \infty$ and $n\leq n',n''$, we have that
\begin{align}\label{e.SLLN.non.random}
\frac{1}{\#(\mathcal{A}_{n'}\times \mathcal{A}_{n''})}\sum_{i\neq j}(-\log|c_i-c_j|)f(c_i)f(c_j)\to I(f(x)\phi(x)\,dx),
\end{align}
where the sum is taken over $(i,j)\in \mathcal{A}_{n'}\times \mathcal{A}_{n''}$ and $i\neq j$.
Moreover,
\begin{align*}
I(f(x)\phi(x)\,dx) <\infty.
\end{align*}
\end{lemma}

\begin{remark}
The density function $\phi(x)$ in assumption \ref{a.distributed} is not to be confused with the continuous density function $f$ in Definition \ref{d.TEA} and in \ref{property}. Once the centers are distributed with respect to $\phi(x)$, the function $\phi(x)$ is fixed. The continuous density function $f$ in Definition \ref{d.TEA} and in \ref{property} is arbitrary.
\end{remark}

\begin{proof}[Proof of Lemma \ref{l.SLLN.non.random}]
Given $\eps>0$, let $\delta>0$ satisfy assumption \ref{a.summable}. For $s>0$, define $f_s(x)=-\log x$ for $x\geq s$ and $0$ otherwise. Using Fatou's lemma, we get
\begin{align*}
\iint_{|x-y|<\delta} (-\log|x-y|)\phi(x)\phi(y) \,dx\,dy \leq& \liminf_{s\rightarrow 0^+} \iint_{|x-y|<\delta} f_s(|x-y|)\phi(x)\phi(y) \,dx\,dy.
\end{align*}
Using assumption \ref{a.distributed} for $n',n''\in \{n,2n,\dots, 2^{q(n)}n \},$ we have 
\begin{align*}
\liminf_{s\rightarrow 0^+} \iint_{|x-y|<\delta} f_s(|x-y|)\phi(x)\phi(y) \,dx\,dy=&\liminf_{s\rightarrow 0^+} \lim_{n\rightarrow\infty}\sum_{0<|c_i-c_j|<\delta}\frac{f_s(|c_i-c_j|)}{\#(\mathcal{A}_{n'}\times \mathcal{A}_{n''})} \\
\leq& \lim_{n\rightarrow\infty} \sum_{0<|c_i-c_j|<\delta} \frac{(-\log|c_i-c_j|)}{\#(\mathcal{A}_{n'}\times \mathcal{A}_{n''})}\\
\leq& \eps,\end{align*}
where the last holds by assumption \ref{a.summable} for some $\delta>0$. Hence, for every continuous function $f$, we have finite energy:
\begin{align*}
I(f(x)\phi(x)\,dx)=\iint (-\log|x-y|)f(x)\phi(x)f(y)\phi(y) \,dx\,dy <\infty.
\end{align*}
Note that 
\begin{align*}
\sum_{i\neq j}\frac{(-\log|c_i-c_j|)f(c_i)f(c_j)}{\#(\mathcal{A}_{n'}\times \mathcal{A}_{n''})}-&\sum_{i\neq j}\frac{(-\log|c_i-c_j|)f_\delta(c_i)f_\delta(c_j)}{\#(\mathcal{A}_{n'}\times \mathcal{A}_{n''})}\\
&=\sum_{0<|c_i-c_j|<\delta}\frac{(-\log|c_i-c_j|)f(c_i)f(c_j)}{\#(\mathcal{A}_{n'}\times \mathcal{A}_{n''})}.
\end{align*}
Let $n\rightarrow\inf$, then by assumptions \ref{a.distributed} and \ref{a.summable}, we have

\begin{align*}
\left|\lim_{n\rightarrow\infty}\sum_{i\neq j}\frac{(-\log|c_i-c_j|)f(c_i)f(c_j)}{\#(\mathcal{A}_{n'}\times \mathcal{A}_{n''})}-\iint f_\delta(|x-y|)f(x)f(y)\phi(x)\phi(y)\right|\\
\leq\eps \cdot(\max |f|)^2.
\end{align*}
By letting  $\delta\rightarrow 0$, we have that
\begin{align*}
\left|\lim_{n\rightarrow\infty}\sum_{i\neq j}\frac{(-\log|c_i-c_j|)f(c_i)f(c_j)}{\#(\mathcal{A}_{n'}\times \mathcal{A}_{n''})}-I(f(x)\phi(x)\,dx)\right|
\leq\eps \cdot\max |f|.
\end{align*}
Since $\eps>0$ is arbitrary, we get \eqref{e.SLLN.non.random}.

\end{proof}

\subsection{Construction of a single-level re-distribution}\label{s.construction}
Our first step in constructing the probability measures in \ref{property} is to construct a \textit{single-level re-distribution} probability measure. This section is devoted to the construction of such probability measures.

We begin with a ``re-distribution" type of measure $f(x)\left.dx\right|_{[0,1]}$ onto a single interval:
\begin{align*}
\mu_k=\frac{f(x)\left.dx\right|_{I_k}}{|I_k|}.
\end{align*}
We do not call this a re-distribution as in \cite{KQ} because the measure is not necessarily a probability measure. We consider the average of $\mu_k$'s:
\begin{align*}
\mu_{\mathcal{A}_n}:=\frac{1}{\#\mathcal{A}_n}\sum_{k\in \mathcal{A}_n}\mu_k,
\end{align*}
where $\mathcal{A}_n$ are defined in \eqref{d.A_n} and satisfy the gap control property \ref{a.uniform}. Recall that gap control implies that for large enough $n$, our collection of intervals are disjoint (see Remark \ref{r.I_k.are.disjoint}). Notice that each measure $\mu_{\mathcal{A}_n}$ is not necessarily a probability measure. To correct that, let us define
$$
V_n=\bigcup_{k\in\mathcal{A}_n} I_k.
$$
Then, we consider the \textit{single-level re-distribution} probability measure:
\begin{align}\label{d.mu.hat}
\hat{\mu}_n:=\frac{\mu_{\mathcal{A}_n}}{\mu_{\mathcal{A}_n}(V_n)},
\end{align}
which is supported on $V_n$ and it's energy is
\begin{align*}
I(\hat{\mu}_n)=\frac{1}{(\mu_{\mathcal{A}_n}(V_n))^2}I(\mu_{\mathcal{A}_n}).
\end{align*}

Thus, we are interested in the asymptotic behavior of $\mu_{\mathcal{A}_n}(V_n)$ and the asymptotic behavior of:
\begin{align}\label{e.partition.energy}
I\left(\mu_{\mathcal{A}_n}\right)=\frac{1}{(\#\mathcal{A}_n)^2}\sum_{k\in\mathcal{A}_n} I(\mu_k) +\frac{1}{(\#\mathcal{A}_n)^2}\sum_{i\neq j}I(\mu_{i},\mu_{j}).
\end{align}

The first sum is referred to as the \textit{self-interaction} sum because in $I(\mu_k)=I(\mu_k,\mu_k)$ the same measure is interacting with itself. The second sum is referred to as \textit{outer-interaction} sum because we have two measures with disjoint supports interacting with each other in $I(\mu_k,\mu_j)$.

In Section \ref{s.outer}, we will discuss the asymptotic behavior of the outer-interaction and in Section \ref{s.self} we will work on controlling the asymptotic behavior of the self-interaction sum. In Section \ref{s.critical}, we will put the two together. We will finish the section with the asymptotic behavior of $\mu_{\mathcal{A}_n}(V_n):$

\begin{lemma} If \ref{a.distributed} and \ref{a.uniform} hold, then 
\begin{align*}
\mu_{\mathcal{A}_n}(V_n)=\int f(x)\phi(x)\,dx +o(1).
\end{align*}\label{l.normalization}

\end{lemma}
\begin{proof} Since for large $n$ the intervals in
$$\{I_k:k\in\mathcal{A}_n\}$$
are disjoint (see Remark \ref{r.I_k.are.disjoint}), then $\mu_k(V_n)=\frac{1}{|I_k|}\int_{I_k}f(x)\,dx.$
Hence,
\begin{align*}
\mu_{\mathcal{A}_n}(V_n)=\frac{1}{\#\mathcal{A}_n}\sum_{k\in \mathcal{A}_n}\frac{1}{|I_k|}\int_{I_k}f(x)\,dx.
\end{align*}
Due to the uniform continuity of $f$ on the interval $[0,1]$, for every $\eps>0$, there exists $\delta>0$ such that 
$$|f(x)-f(y)|<\eps \quad \text{ if }\quad |x-y|<\delta.$$
 Since the lengths of the intervals $I_k$ approach 0, then there exists $N\in\N$ such that for every $k\geq N$ we have 
$$|f(x)-f(c_i)|<\eps \quad \text{ if }\quad x\in I_k.$$
Therefore, for every $n\geq N$ and every $k\in \mathcal{A}_n$, we have 
$$|f(x)-f(c_k)|<\eps \quad \text{ if }\quad x\in I_k.$$
It follows that for large enough $n$, we have
\begin{align*}
\left|\mu_{\mathcal{A}_n}(V_n)-\frac{1}{\#\mathcal{A}_n}\sum_{i\in \mathcal{A}_n}f(c_i)\right|<\eps.
\end{align*}
As the centers are distributed with respect to $\phi(x)$ \ref{a.distributed}, it follows that
\begin{align*}
\left|\mu_{\mathcal{A}_n}(V_n)-\left(\int f(x)\phi(x)\,dx+o(1)\right)\right|<\eps.
\end{align*}
Since $\eps>0$ was arbitrary, then the result holds.
\end{proof}


\subsection{Asymptotic behavior of outer-interaction}\label{s.outer}
In this section, we are interested in the asymptotic behavior of the outer-interaction sum:
\begin{align*}
\frac{1}{(\#\mathcal{A}_n)^2}\sum_{i\neq j}I(\mu_{i},\mu_{j}),
\end{align*}
where the sum is taken over $i,j\in\mathcal{A}_n$ and $i\neq j$. It is the outer-interaction sum that gives the limit point in \ref{property}:

\begin{lemma}\label{l.outer}
We have 
\begin{align}
\frac{1}{(\#\mathcal{A}_n)^2}\sum_{i\neq j}I(\mu_{i},\mu_{j})=I(f(x)\phi(x)dx)+o(1),
\end{align}
where $i,j\in\mathcal{A}_n$.
\end{lemma}

To show Lemma \ref{l.outer}, we want to estimate $(-\log |x-y|)$ by $(-\log |c-c'|)$ where $x$ and $y$ are in intervals with centers $c$ and $c'$, respectively. The next lemma allows us to do that.

\begin{lemma}\label{l.estimate.log}
Let $f$ be any continuous function on $[0,1]$ and let $J,J'$ be two disjoint intervals in $[0,1]$ with  lengths $r,r'$ and centers $c,c'$, respectively. Define
\begin{align*}
\mu:= \frac{1}{r}\left.f(x)\,dx\right|_J \quad \text{and}\quad \mu':= \frac{1}{r'}\left.f(x)\,dx\right|_{J'}. 
\end{align*}
Let $\eps>0$. If
\begin{align}\label{e.estimate.log.condition}
\frac{r+r'}{2|c-c'|}\leq (1-e^{-\eps}),
\end{align}
then
\begin{align*}
\left|I(\mu,\mu')-(-\log |c-c'|)f(c)f(c')\right|\leq(2K(-\log |c-c'|)+K^2)\eps,
\end{align*}
where $K=\norm{f}_\infty$.
\end{lemma}

\begin{proof}
Let us first prove that for $a,b>0$, we have
\begin{align*}
|\log a-\log b|\leq \eps,
\end{align*}
if 
\begin{align*}
|a-b|\leq b(1-e^{-\eps}).
\end{align*}
We have that 
\begin{align*}
|\log a-\log b|\leq \eps
\end{align*}
holds if and only if 
\begin{align*}
-\eps\leq \log (a/b)\leq \eps,
\end{align*}
if and only if
\begin{align*}
be^{-\eps}\leq a\leq be^\eps,
\end{align*}
if and only if
\begin{align*}
-b(1-e^{-\eps})\leq a-b\leq b(e^\eps-1).
\end{align*}
Since $(1-e^{-\eps})\leq (e^\eps-1)$, then
\begin{align*}
|\log a-\log b|\leq \eps
\end{align*}
holds when 
\begin{align*}
|a-b|\leq b(1-e^{-\eps}).
\end{align*}

For any two disjoint intervals $J,J'\in [0,1]$ with centers $c,c'$ and with lengths $r,r'$ respectively, we have 
\begin{align*}
||x-y|-|c-c'||\leq|(x-c)+(c'-y)|\leq \frac{r+r'}{2}.
\end{align*}

Since assumption \eqref{e.estimate.log.condition} is equivalent to
\begin{align*}
\frac{r+r'}{2}\leq |c-c'|(1-e^{-\eps}),
\end{align*}
then
\begin{align*}
|(-\log |x-y|)-(-\log |c-c'|)|<\eps.
\end{align*}

Since $f$ is uniformly continuous on $[0,1]$, then there exists $\delta>0$ such that
$|f(a)-f(b)|<\eps$ if $|a-b|<\delta$. If $0<r,r'<\delta$, we have that $|f(x)-f(c)|<\eps$ and $|f(y)-f(c')|<\eps.$ We would like to combine the three inequalities. 

Suppose $A,a,B,b\in\R$ and $\eps_a,\eps_b>0$ such that
\begin{align*}
|A-a|<\eps_a \quad \text{and} \quad |B-b|<\eps_b.
\end{align*}
We have that 
\begin{align*}
|AB-ab|\leq |AB-Ab|+|Ab-ab|\leq (|A|\eps_b+|b|\eps_a).
\end{align*}
One application of the above gives
\
\begin{align*}
|f(x)f(y)-f(c)f(c')|\leq 2K\eps,
\end{align*}
where $K=\norm{f}_\infty$.
A third application yields
\begin{align*}
|(-\log |x-y|)f(x)f(y)-(-\log |c-c'|)f(c)f(c')|<(2K\eps(-\log |c-c'|)+K^2\eps).
\end{align*}
Integrating by 
\begin{align*}
 \frac{1}{r}\left.dx\right|_J \quad \text{and}\quad  \frac{1}{r'}\left.dy\right|_{J'} 
\end{align*}
finishes the proof.
\end{proof}

Now that we can estimate $(-\log |x-y|)$ by $(-\log |c-c'|)$ where $x$ and $y$ are in intervals with centers $c$ and $c'$, respectively, we are ready to estimate
\begin{align*}
\frac{1}{(\#\mathcal{A}_n)^2}\sum_{i\neq j}I(\mu_{i},\mu_{j}),
\end{align*}
by
\begin{align*}
\frac{1}{(\#\mathcal{A}_n)^2}\sum_{i\neq j}(-\log |c_i-c_j|)f(c_i)f(c_j).
\end{align*}
Let us go back to Lemma \ref{l.outer} and prove the asymptotic behavior of the outer-interaction:

\begin{proof}[Proof of Lemma~\ref{l.outer}]
Let $\epsilon>0$ be given and let $K=\norm{f}_\infty$. The gap control \ref{a.uniform} guarantees that there exists $N\in \N$ such that for every $n\geq N$ and every $i\neq j$, where $i,j\in\mathcal{A}_n$, we have 
\begin{align*}
\frac{l_i+l_j}{2|c_i-c_j|}\leq (1-e^{-\eps}),
\end{align*}
which is the condition \eqref{e.estimate.log.condition} in Lemma \ref{l.estimate.log}. Since $l_k$ decrease to $0$ as $k\rightarrow \infty$, then for all large enough $n$ we may apply Lemma \ref{l.estimate.log} to get
\begin{align*}
|I(\mu_{i},\mu_{j})-(-\log |c_i-c_j|)f(c_i)f(c_j)|\leq (2K(-\log |c_i-c_j|)+K^2)\eps,
\end{align*}
for every $n\geq N$ and every $i\neq j$, where $i,j\in\mathcal{A}_n$. Adding this up for $i\neq j$ where $i,j\in\mathcal{A}_n$ and then dividing by $(\#\mathcal{A}_n)^2$, gives us:
\begin{align*}
\left|\frac{1}{(\#\mathcal{A}_n)^2}\sum_{i\neq j}I(\mu_{i},\mu_{j}) \right.
-\left.\frac{1}{(\#\mathcal{A}_n)^2}\sum_{i\neq j} (-\log |c_i-c_j|)f(c_i)f(c_j)\right|\\
\leq\frac{2K\eps}{(\#\mathcal{A}_n)^2}\sum_{i\neq j} (-\log |c_i-c_j|)+K^2\eps.
\end{align*}
We will apply Lemma \ref{l.SLLN.non.random} twice to the last two sums. We apply the lemma to the last sum by taking $f=1$ in Lemma \ref{l.SLLN.non.random}, and then we apply the lemma again for arbitrary $f$ in Lemma \ref{l.SLLN.non.random} to get:
\begin{align*}
\left|\lim_{n\rightarrow\infty}\frac{1}{(\#\mathcal{A}_n)^2}\sum_{i\neq j}I(\mu_{i},\mu_{j})-I(f(x)\phi(x)dx)\right|\\
\leq(2KI(\phi(x)dx)+K^2)\eps.
\end{align*}
Lemma \ref{l.SLLN.non.random} also informs us that the energy of $\phi(x)dx$ is finite, hence $0<(2KI(\phi(x)dx)+K^2)\eps<\infty$. As  $\eps>0$ is arbitrary, it follows that
\begin{align*}
\frac{1}{(\#\mathcal{A}_n)^2}\sum_{i\neq j}(-\log |c_i-c_j|)f(c_i)f(c_j)\rightarrow  I(f(x)\phi(x)dx).
\end{align*}
Lemma~\ref{l.outer} holds.

\end{proof}

\subsection{Asymptotic behavior of self-interaction} \label{s.self}
In this section, we will to control the self-interaction:
\begin{lemma} If $l_k= e^{-\lambda k}$ where $\lambda>0$, then 
\begin{align*}
\frac{1}{(\#\mathcal{A}_n)^2}\sum_{k\in\mathcal{A}_n} I(\mu_k) =O\left(  \frac{1}{(\#\mathcal{A}_n)^2}\sum_{k\in\mathcal{A}_n} k+o(1)\right).
\end{align*}
\end{lemma}
\begin{proof}
By shifting and a change of variables, we get
\begin{align*}
I(\frac{1}{l_k} dx|_{I_{i}}) = -\log l_k + I(\left. dx\right|_{[0,1]}) = -\log l_k \cdot (1+o(1)).
\end{align*}
If $l_k= e^{-\lambda k}$, then adding the above over $k\in\mathcal{A}_n$ gives us our result.

\end{proof}
If the self-interaction sum vanishes in the limit, then we will be able to finish the proof with a single-level re-distribution. Let us see what the self-interaction tells us:

\begin{lemma}\label{l.self.constant.one} If $\mathcal{A}_n:=\{n,\ldots, p(n)-1\}$ where $p(n)$ is an integer-valued function such that $p(n)\geq 2n$, then
\begin{align*}
\frac{1}{(\#\mathcal{A}_n)^2}\sum_{k\in\mathcal{A}_n} k=\frac{p(n)+n-1}{2(p(n)-n)}.
\end{align*}
If $p(n)>>n$, then the right-hand side is close to $\frac{1}{2}$.
\end{lemma}
\begin{proof}
Let 
\begin{align*}
S=\sum_{k\in\mathcal{A}_n} k.
\end{align*}
Then the arithmetic sum becomes
$$S=\frac{(p(n)+n-1)(\#\mathcal{A}_n)}{2}. $$
Since $\#\mathcal{A}_n= p(n)-n$, then dividing by $(\#\mathcal{A}_n)^2$ we get what we want.
\end{proof}

\begin{remark}\label{r.single.level.not.enough}
Lemma \ref{l.self.constant.one} tells us that no matter how many intervals are included in our single-level of re-distribution, the self-interaction sum will never vanish. 

But since we can bound the self-interaction sum uniformly for all $n$, we will be able to apply a multi-level re-distribution in Section \ref{s.critical}. That is, we will take the average of measures $\hat{\mu}_n$ to handle the self-interaction sum.
\end{remark}


\subsection{The proof of Proposition \ref{p.S.alpha=1.total.energy}}\label{s.critical}

In this section, we will use a multi-level re-distribution to show \ref{property} holds and prove Proposition \ref{p.S.alpha=1.total.energy}.

Let us first see where the asymptotic behavior of a single-level re-distribution leads:

\begin{proposition}[Single-level re-distribution]\label{p.aysmptotic.mu.critical.case} 
Let $l_k:= e^{-\lambda k}$ and $\lambda>0$. If assumptions \ref{a.distributed}-\ref{a.uniform} are satisfied, then
\begin{align*}
I(\hat{\mu}_n)=& \frac{I(f(x)\phi(x)\,dx)}{(\int f(x)\phi(x)\,dx)^2} +O\left(\frac{1}{(\int f(x)\phi(x)\,dx)^2}\right)+o(1)=O(1) ,
\end{align*}
where each $\hat{\mu}_n$ is the corresponding measure defined in \eqref{d.mu.hat}.
\end{proposition}
\begin{proof} We have that
\begin{align*}
I(\hat{\mu}_n)=\frac{1}{(\mu_{\mathcal{A}_n}(V_n))^2}I(\mu_{\mathcal{A}_n}).
\end{align*}
Breaking down the last, we get
\begin{align*}
I\left(\mu_{\mathcal{A}_n}\right)=\frac{1}{(\#\mathcal{A}_n)^2}\sum_{k\in\mathcal{A}_n} I(\mu_k) +\frac{1}{(\#\mathcal{A}_n)^2}\sum_{i\neq j}I(\mu_{i},\mu_{j}).
\end{align*}
Applying Lemma \ref{l.outer} to the outer-interaction sum and Lemma \ref{l.self.constant.one} to the self-interaction sum, and lastly, applying Lemma \ref{l.normalization} to the normalization $\mu_{\mathcal{A}_n}(V_n)$ completes the proof.

\end{proof}

Remark \ref{r.single.level.not.enough} tells us that using a single level re-distribution will not render $S$ to be re-distributable no matter how many intervals are included in $\{I_k:k\in\mathcal{A}_n\}$. We will need to take the average of $q(n)$ single-level re-distribution measures, where $q(n)$ is an integer-valued function such that $q(n)\rightarrow\infty$ as $n\rightarrow \infty$. 

Let $\hat{\mu}_n$ be a single-level re-distribution as defined in \eqref{d.mu.hat}. For each $n$, we consider a \textit{multi-level re-distribution} probability measure:
\begin{align*}
\mu^m:=\frac{1}{\#\mathcal{B}_m}\sum_{s\in \mathcal{B}_m} \hat{\mu}_{2^{s}m},
\end{align*}
where $$\mathcal{B}_m:=\{0,\ldots,{q(m)-1}\}.$$
Since each $\hat{\mu}_n$ is supported on
\begin{align*}
V_n:=\bigcup_{k\in\mathcal{A}_n}I_k=\bigcup_{k=n}^{2n-1}I_k,
\end{align*}
then $\mu^m$ is supported on
\begin{align*}
V_n, V_{2n},V_{2^2n},\ldots ,V_{2^{q(n)-1}n}.
\end{align*}
See Remark \ref{r.I_k.are.disjoint} for details.

Our convex measure can now be partitioned into a new self-interaction sum and a new outer-interaction sum:
\begin{align*}
I(\mu^m)=\frac{1}{(\#\mathcal{B}_m)^2}\sum_{s\in \mathcal{B}_m} I(\hat{\mu}_{2^{s}m})+\frac{1}{(\#\mathcal{B}_m)^2}\sum_{\substack{s,t\in \mathcal{B}_m\\ s\neq t}} I(\hat{\mu}_{2^{s}m},\hat{\mu}_{2^{t}m}).
\end{align*}
Proposition \ref{p.aysmptotic.mu.critical.case} tells us that 
$$I(\hat{\mu}_{n})=O(1).$$
Hence,
\begin{align*}
\frac{1}{(\#\mathcal{B}_m)^2}\sum_{s\in \mathcal{B}_m} I(\hat{\mu}_{2^{s}m})=&\frac{1}{(\#\mathcal{B}_m)^2}\sum_{s\in \mathcal{B}_m} O(1)\\
\leq& \frac{O(1) }{(\#\mathcal{B}_m)}\rightarrow 0\quad \text{ as }\quad m\rightarrow\infty.
\end{align*}
That is, the self-interaction sum vanishes. The outer-interaction sum gives what we aim:

\begin{lemma} Assume \ref{a.distributed}-\ref{a.uniform}. As $m\rightarrow\infty$ we have that \label{l.outer.critical}
\begin{align*} 
\frac{1}{(\#\mathcal{B}_m)^2}\sum_{s\neq t} I(\hat{\mu}_{2^{s}m},\hat{\mu}_{2^{t}m})\rightarrow \frac{I(f(x)\phi(x)\,dx)}{(\int f(x)\phi(x)\,dx)^2},
\end{align*}
where the sum is over $s\neq t$ and $s,t\in \mathcal{B}_m$.
\end{lemma}
We will leave the proof of Lemma \ref{l.outer.critical} to the end of the section. Note that the vanishing of the self-interaction sum and Lemma \ref{l.outer.critical} gives us:

\begin{proposition}\label{p.total.energy.critical} For every $f\in C([0,1])$, we have that
\begin{align*}
I(\mu^m)=& \frac{I(f(x)\phi(x)\,dx)}{(\int f(x)\phi(x)\,dx)^2}+o(1).
\end{align*}
\end{proposition}

Notice that with Proposition \ref{p.total.energy.critical} we can show that $S$ is re-distributable when $\phi(x)\equiv1$. In order to remove $\phi(x)$, we will need to apply Proposition \ref{p.total.energy.critical} to a continuous approximation of $1/\phi(x)$ and take an appropriate subsequence:

\begin{proposition} \label{p.absolutely.continuous}
Suppose for each continuous function $f$, there exists a sequence of probability measures $\{\mu^n\}$ so that each $\mu^n$ has a piecewise continuous density with
support $V_n$, where $V_n$ is a finite union of disjoint $I_k$'s with $k\geq n$ with asymptotic behavior:
\begin{align}\label{p.asymptotic.distribution.assumption}
I(\mu^n)=& \frac{I(f(x)\phi(x)dx)}{(\int f(x)\phi(x)\,dx)^2} +o(1).
\end{align}
Then property \ref{property} holds.

\end{proposition}

Let us go back to show that $S$ is re-distributable using a multi-level re-distribution before we prove Proposition \ref{p.absolutely.continuous}.

\begin{proof}[Proof of Proposition \ref{p.S.alpha=1.total.energy} ] Since $\phi(x)>0$ almost everywhere, then combining Proposition \ref{p.total.energy.critical} and Proposition \ref{p.absolutely.continuous} shows \ref{property} holds.

Given any probability measure $\nu$ with piecewise continuous density that is supported on a finite collection of intervals in $[0,1]$ and given any $\eps>0$, we may apply \ref{property} to a continuous $L^1$ approximation of the density function of $\nu$ to show that there exists a probability measure $\nu'$ satisfying properties (1) and (2) in Definition \ref{d.TEA}. Thus, $S$ is re-distributable.

\end{proof}

Now, let us analyze the asymptotic behavior of the new outer-interaction sum:

\begin{proof}[Proof of Lemma \ref{l.outer.critical} ] 
The goal is to show that 
\begin{align*} 
I(\hat{\mu}_{n},\hat{\mu}_{n'})=\frac{I(\mu_{\mathcal{A}_n},\mu_{\mathcal{A}_{n'}})}{\mu_{\mathcal{A}_n}(V_n)\cdot\mu_{\mathcal{A}_{n'}}(V_{n'})}\rightarrow \frac{I(f(x)\phi(x)\,dx)}{(\int f(x)\phi(x)\,dx)^2},
\end{align*}
as $m\rightarrow \infty$ and independently of our choice of $n,n'\in\{2^{s}m:s\in\mathcal{B}_m\}$. Once we accomplish this, then 
\begin{align*} 
\frac{1}{(\#\mathcal{B}_m)^2}\sum_{s\neq t} I(\hat{\mu}_{2^{s}m},\hat{\mu}_{2^{t}m})\rightarrow \frac{I(f(x)\phi(x)\,dx)}{(\int f(x)\phi(x)\,dx)^2},
\end{align*}
as $m\rightarrow\infty$. 

Lemma \ref{l.normalization} shows that
\begin{align*} 
\frac{1}{\mu_{\mathcal{A}_n}(V_n)\cdot\mu_{\mathcal{A}_{n'}}(V_{n'})}\rightarrow \frac{1}{(\int f(x)\phi(x)\,dx)^2},
\end{align*}
as $m\rightarrow \infty$ and independently of our choice of $n,n'\in\{2^{s}m:s\in\mathcal{B}_m\}$. 

Let us focus on $I(\mu_{\mathcal{A}_n},\mu_{\mathcal{A}_{n'}})$. Given $n\neq n'$ where $n,n'\in\{2^{s}m:s\in\mathcal{B}_m\}$,  we have that
\begin{align}
{I(\mu_{\mathcal{A}_n},\mu_{\mathcal{A}_{n'}})}=\frac{1}{(\#\mathcal{A}_n)(\#\mathcal{A}_{n'})}\sum_{i\neq j} I(\mu_i,\mu_j),\label{e.uniform.limit.energy} 
\end{align}
where $(i,j)\in\mathcal{A}_n\times\mathcal{A}_{n'}$.  By the gap control assumption \ref{a.uniform}, we know that for all large enough $m$ and every $i\neq j$, where $i,j\in\{m,\ldots,{2^{q(m)}m-1}\}$, we have 
\begin{align*}
\frac{l_i+l_j}{2|c_i-c_j|}\leq (1-e^{-\eps}),
\end{align*}
which is the needed condition \eqref{e.estimate.log.condition} to apply Lemma \ref{l.estimate.log} for all large enough $m$. Lemma \ref{l.estimate.log} gives us
\begin{align*}
|I(\mu_{i},\mu_{j})-(-\log |c_i-c_j|)f(c_i)f(c_j)|\leq (2K(-\log |c_i-c_j|)+K^2)\eps,
\end{align*}
for every $i\neq j$ where $i,j\in\{m,\ldots,{2^{q(m)}m-1}\}$ and for all large enough $m$. Adding this up over $(i,j)\in\mathcal{A}_n\times\mathcal{A}_{n'}$ and then dividing by $(\#\mathcal{A}_n)(\#\mathcal{A}_{n'})$ gives us:

\begin{align*}
\left|\frac{1}{(\#\mathcal{A}_n)(\#\mathcal{A}_{n'})}\sum_{i\neq j}I(\mu_{i},\mu_{j}) \right.
-\left.\frac{1}{(\#\mathcal{A}_n)(\#\mathcal{A}_{n'})}\sum_{i\neq j} (-\log |c_i-c_j|)f(c_i)f(c_j)\right|\\
\leq\frac{2K\eps}{(\#\mathcal{A}_n)(\#\mathcal{A}_{n'})}\sum_{i\neq j} (-\log |c_i-c_j|)\\
+K^2\eps.
\end{align*}
We remark that the inequality holds independently of our choice of $n,n'\in\{2^{s}m:s\in\mathcal{B}_m\}$ for all large enough $m$. By the distribution assumption \ref{a.distributed} and the $\log$-average spacing assumption \ref{a.summable}, we can apply Lemma \ref{l.SLLN.non.random} twice to the last two sums above. One application yields:
\begin{align*}
\frac{1}{(\#\mathcal{A}_n)(\#\mathcal{A}_{n'})}\sum_{i\neq j} (-\log |c_i-c_j|)f(c_i)f(c_j)\rightarrow I(f(x)\phi(x)\,dx),
\end{align*}
as $m\rightarrow \infty$ with $n,n'\geq m$. For the second application we take $f=1$ in Lemma \ref{l.SLLN.non.random} to get
\begin{align*}
\frac{1}{(\#\mathcal{A}_n)^2}\sum_{i\neq j} (-\log |c_i-c_j|)\rightarrow I(\phi(x)\,dx)< \infty,
\end{align*}
as $m\rightarrow \infty$ with $n,n'\geq m$.
Therefore, for $n,n'\in\{2^sm:s\in\mathcal{B}_m\}$, we have
\begin{align*}
\left|\lim_{m\rightarrow\infty} I(\mu_{\mathcal{A}_n},\mu_{\mathcal{A}_{n'}}) 
-I(f(x)\phi(x)\,dx) \right|\\
\leq{2K\eps}I(\phi(x)\,dx) +K^2\eps.
\end{align*}
Since $\eps>0$ is arbitrary and $I(\phi(x)\,dx)<\infty$, then for $n,n'\in\{2^sm:s\in\mathcal{B}_m\}$, we have
\begin{align*}
\left|\lim_{m\rightarrow\infty} I(\mu_{\mathcal{A}_n},\mu_{\mathcal{A}_{n'}}) 
-I(f(x)\phi(x)\,dx)\right|=0.
\end{align*}
Therefore, as $m\rightarrow \infty$, then

\begin{align*} 
I(\hat{\mu}_{n},\hat{\mu}_{n'})\rightarrow \frac{I(f(x)\phi(x)\,dx)}{(\int f(x)\phi(x)\,dx)^2},
\end{align*}
where $n,n'\in\{2^sm:s\in\mathcal{B}_m\}$, which completes the proof. 

\end{proof}

\begin{proof}[Proof of Proposition \ref{p.absolutely.continuous} ] 

For every continuous function $h$, set
\begin{align*}
f(x):=\frac{h(x)}{\phi(x)},
\end{align*}
when $\phi\neq 0$ and $0$ otherwise. For each $\eps>0$, there exists a continuous function $f'$ such that 
\begin{align*}
\left| \frac{I(f_0(x)\phi(x)\,dx)}{(\int_0^1 f_0(x)\phi(x)\,dx)^2}-  \frac{I(f(x)\phi(x)\,dx)}{(\int_0^1 f(x)\phi(x)\,dx)^2}\right|<\eps/2.
\end{align*}
Applying \eqref{p.asymptotic.distribution.assumption} to $f_0$, gives us that for each $\eps>0$, there exists $N$ such that for every $n\geq N$ we have
\begin{align*}
\left| \frac{I(f_0(x)\phi(x))}{(\int_0^1 f_0(x)\phi(x)\,dx)^2}-  I(\mu^n)\right|<\eps/2,
\end{align*}
where each $\mu^n$ is a probability measure with a piecewise continuous density with support in $V_n$, where $V_n$ is a finite unions of disjoint $I_k$'s with $k\geq n$. Since $f(x)\phi(x)=h(x)$ a.e., then for every $\eps>0$, there exists $n\in\N$ as large as we need such that 
 
\begin{align*}
\left|I({\mu}^n)-  \frac{I(h(x)\,dx)}{(\int_0^1 h(x)\,dx)^2}\right|\leq& \left|I(\mu^n)- \frac{I(f_0(x)\phi(x))}{(\int_0^1 f_0(x)\phi(x)\,dx)^2}\right|\\
&+\left| \frac{I(f_0(x)\phi(x))}{(\int_0^1 f_0(x)\phi(x)\,dx)^2}-  \frac{I(f(x)\phi(x))}{(\int_0^1 f(x)\phi(x)\,dx)^2}\right|\\
<&\eps.
\end{align*}

Hence, there exists a subsequence $n_k$ and probability measures $\mu^{n_k}$ with piecewise continuous density   supported in $V_{n_k}$ such that
\begin{align}\label{e.concl.asymptotic}
I({\mu}^{n_k})=& \frac{I(h(x)\,dx)}{(\int_0^1 h(x)\,dx)^2}+o(1).
\end{align}
Each $\mu^{n_k}$ is a probability measure with a piecewise continuous density with support contained in $V_{n_k}$, that is a finite union of disjoint $I_j$'s with $j\geq k$. Hence, these $\{\mu^{n_k}\}$ satisfy \ref{property}.

\end{proof}

\section*{Acknowledgements}
I would like to thank Victor Kleptsyn for numerous discussions and for the idea of the proof of Lemma \ref{l.A2sat}, as well as for reading multiple drafts
of this paper and offering many helpful suggestions, and Anton Gorodetski for his remarks. 


\end{document}